\documentclass[11pt]{amsart}

\usepackage{a4wide,latexsym,amssymb,amsthm,amsmath}
\usepackage{url}

\theoremstyle{plain}

\newtheorem{theorem}{Theorem}
\newtheorem{lemma}{Lemma}
\newtheorem{proposition}{Proposition}

\newtheorem{remark}{Remark}
\newtheorem{example}{Example}

\def\dt{\partial_t}
\def\dx{\partial_x}
\def\dy{\partial_y}
\def\<{\langle}
\def\>{\rangle}

\def \n {\nabla}
\def \b {\beta}
\def \a {\alpha}

\def \Mt {\widetilde{M}}

\def\<{ \left < }
\def\>{ \right > }

\def \f {\mathbf{f}}

\def\R{\mathbb{R}}

\def\H{\mathbb{H}}
\def\S{\mathbb{S}}

\begin{document}

\title[]
{Canonical Coordinates and Principal Directions for Surfaces in $\H^2\times\R$}
\thanks{{\bf This paper will appear in Taiwanese Mathematical Journal, 2010.}}

\author[F. Dillen]{Franki Dillen}

\address[F. Dillen {\rm and} A.~I. Nistor]{Katholieke Universiteit Leuven\\ Departement
Wiskunde\\ Celestij\-nenlaan 200 B\\ B-3001 Leuven\\ Belgium}
\email[F. Dillen]{franki.dillen (at) wis.kuleuven.be}
\email[A.~I. Nistor]{ana.irina.nistor (at) gmail.com}

\author[M.~I. Munteanu]{Marian Ioan Munteanu}

\address[M.~I. Munteanu]{University 'Al. I. Cuza' Ia\c si\\ Faculty of Mathematics\\
Bd. Carol I, no.11\\ 700506 Ia\c si\\ Romania\\
\url{http://www.math.uaic.ro/~munteanu}}
\email[M.~I. Munteanu]{marian.ioan.munteanu (at) gmail.com}

\author[A.~I. Nistor]{Ana-Irina Nistor}

\thanks{This research was supported by Research Grant G.0432.07 of the Research Foundation-Flanders (FWO).
The second author was supported by Grant PN-II ID 398/2007-2010 (Romania).
}

\begin{abstract}
In this paper we characterize and classify surfaces in
${\mathbb{H}}^2\times{\mathbb{R}}$ which have a canonical principal
direction. Here ${\mathbb{H}}^2$ denotes the hyperbolic plane. We
study some geometric properties such as minimality and flatness. Several
examples are given to complete the study.
\end{abstract}


\date{\today}

\maketitle

\section{Introduction}

The geometry of surfaces in spaces of dimension 3, especially of the form ${\mathbb{M}}^2\times{\mathbb{R}}$,
has enriched in last years. The most interesting situations occur when ${\mathbb{M}}^2$
has constant Gaussian curvature, since in these
cases a lot of classification results are obtained.
One of the first problems, minimality, in these ambient spaces
was studied by H. Rosenberg and W.H. Meeks III in \cite{dmn:MR04,dmn:MR05,dmn:Ros02}.
Inspired by these papers, a generalization for arbitrary dimension, namely for
$\S^n\times\R$ and $\H^n\times\R$, is given in \cite{dmn:Dan,dmn:DFV, dmn:MT, dmn:T09}.
In particular, in \cite{dmn:T09} the author proves independently a higher dimensional version
of Theorem 5 in the present paper.
Some other recent results involving minimality and curvature properties for
surfaces immersed in ambient spaces as $\S^2\times\R$ and $\H^2\times\R$ can be found in
\cite{dmn:AEG07,dmn:FM07,dmn:Hau06,dmn:MMP06,dmn:MO07,dmn:SaE08,dmn:ST05,dmn:ST09}.

Another problem, studied in several very recent papers, consists of
characterization and classification of constant angle surfaces in different 3-dimensional spaces belonging to the
Thurston list, namely in ${\mathbb{E}}^3$,
${\mathbb{S}}^2\times{\mathbb{R}}$,
${\mathbb{H}}^2\times{\mathbb{R}}$, or the Heisenberg group ${\rm Nil}_3$.
A constant angle surface is an orientable surface whose unit normal makes a constant
angle, denoted by $\theta$, with a fixed direction. See for example
\cite{dmn:DFVV07,dmn:DM07,dmn:DM09,dmn:FMV08,dmn:LM10, dmn:MN09}. When the
ambient is of the form ${\mathbb{M}}^2\times{\mathbb{R}}$, a favored
direction is ${\mathbb{R}}$. It is known that for a constant angle
surface in ${\mathbb{E}}^3$, ${\mathbb{S}}^2\times{\mathbb{R}}$ or
in ${\mathbb{H}}^2\times{\mathbb{R}}$, the projection of
$\frac{\partial}{\partial t}$ (where $t$ is the global parameter on
${\mathbb{R}}$) onto the tangent plane of the immersed surface,
denoted by $T$, is a principal direction with the corresponding
principal curvature identically zero. The main topic of the present work is to investigate surfaces in ${\mathbb{H}}^2\times{\mathbb{R}}$
for which $T$ is a principal direction. The study of these surfaces was motivated
by the results obtained in \cite{dmn:DFVV} for the ambient space $\S^2\times\R$.

The structure of this paper is the following.
After we introduce the basic notions in {\em Preliminaries}, we show in the sequel
a way of choosing appropriate coordinates useful in the study of flatness and minimality.
We define also canonical coordinates on a surface for which $T$ is a principal direction.

The last section consists of the main results concerning surfaces with a canonical
principal direction. We give a characterization theorem involving the
normally flatness of these surfaces regarded as codimension $2$ surfaces, immersed in
$\R_1^4=\R_1^3\times\R$. Moreover, we formulate the following classification Theorem~\ref{dmn:thm4} (See also Theorem~\ref{dmn:thm5}):

{\em A surface $M$ isometrically immersed in $\H^2\times\R$ has $T$
as principal direction if and only if the immersion $F$ is given,
up to isometries of the ambient space, by
$F:M\rightarrow\H^2\times\R$,
$$
F(x,y)=\left(A(y)\sinh\phi(x)+B(y)\cosh\phi(x),\ \chi(x)\right)
$$
where $(\phi,\ \chi)$ is a regular curve in $\R^2$, $A$ is a curve in $\S_1^2\subset\R_1^3$,
$B$ is a curve in $\H^2\subset\R_1^3$ orthogonal to $A$ such that
the two speeds $A'$ and respectively $B'$ are parallel. Here
$\S_1^2$ denotes the de Sitter space.}

As a consequence, we retrieve the classification of all constant angle surfaces in $\H^2\times\R$,
obtained in \cite{dmn:DM09}. Finally, we give some other theorems under extra assumptions of minimality or flatness.
We construct a lot of suggestive examples to illustrate our study.

\section{Preliminaries}

Let us fix the notations used in this paper.
By $\Mt= \H^2\times\R$ we denote the ambient space given as the Riemannian
product of the hyperbolic space endowed with the metric $g_H$,
namely $(\H^2(-1), g_H)$, and the one dimensional Euclidean space
endowed with the usual metric. The metric on the ambient space is
given by $\widetilde{g}=g_H+dt^2$, where $t$ is the global
coordinate on $\R$. Then $\dt:=\frac{\partial}{\dt}$ denotes an unit
vector field in the tangent bundle $T(\Mt)$ that is tangent to the
$\R$-direction.
We denote by $\widetilde{R}$ either the curvature tensor
$\widetilde R(X,Y)=[\widetilde\nabla_X,\widetilde\nabla_Y]-\widetilde\nabla_{[X,Y]}$,
or the Riemann-Christoffel tensor on $\Mt$ defined by $\widetilde R(W,Z,X,Y)=\widetilde g(W,\widetilde R(X,Y)Z)$.
One has
\begin{equation*}
\label{dmn:RC} \widetilde{R}(X,Y,Z,W)=-g_H(X_H, W_H)g_H(Y_H,
Z_H)+g_H(X_H, Z_H)g_H(Y_H, W_H)
\end{equation*}
 for any $X,Y,Z,W\in T(\Mt)$ and $X_H$ denotes the projection of $X$ to the tangent space of $\H^2$.

In the sequel we study some important properties of submanifolds $(M,g)\hookrightarrow (\Mt, \widetilde{g})$
isometrically immersed in $\Mt$. Roughly speaking, the metric $g$ on $M$ represents the
restriction of $\widetilde{g}$ to $M$.

For an isometric immersion $F:M\rightarrow \Mt$ we recall the classical Gauss and
Weingarten formulas:

{\bf(G)}\qquad\qquad $\widetilde{\nabla}_X Y=\nabla_X Y + h(X,Y)$

{\bf(W)}\qquad\qquad $\widetilde{\nabla}_X N=-A_N X + \nabla_X^\perp N$

where $X,Y$ are tangent to $M$, $\widetilde{\nabla}$ and $\nabla$
denote the Levi-Civita connections on $\Mt$ respectively on $M$, and $N$
denotes any vector field normal to $M$. Moreover $h$ is a symmetric (1, 2)-tensor
field called the second fundamental form of the submanifold $M$,
$A_N$ is a symmetric (1, 1)-tensor field called the shape operator
associated to $N$ and $\nabla^\perp$ denotes the connection in the normal bundle.

In particular, let us consider an oriented surface $M$ in $\Mt$. If $\xi$ is
the unit normal to $M$ associated with the shape operator $A$, the
following property holds:
$$
\widetilde{g}(h(X,Y),\xi)=g(X,AY)
$$
for any vector fields $X,Y$ tangent to $M$. Taking into account these
considerations, since $\dt$ is unitary, it can be decomposed as
\begin{equation}
\label{dmn:t} \dt=T+\cos \theta \xi
\end{equation}
where $T$ is the projection of $\dt$ on $T(M)$ and $\theta$ is the
\emph{angle function} depending of the point of the surface and
supposed to take values in the interval $[0,\pi]$.

Denoting by $R$ the curvature tensor on $M$, after straightforward computations we are able to write the
fundamental equations of Gauss and Codazzi

{\bf (E.G.)}
$R(X,Y)=AX\wedge AY-X\wedge Y+T^\flat\otimes(X\wedge Y)(T)-(X\wedge Y)(T)^\flat\otimes T$

{\bf (E.C.)}
$\left(\nabla_X A\right)Y-\left(\nabla_Y A\right)X =\cos\theta~(g(X,T)Y-g(Y,T)X)$

where $(X\wedge Y)(Z):=g(X,Z)Y-g(Y,Z)X$, $\flat$ denotes the musical isomorphism flat and
$(\omega\otimes X)(Y)=\omega(Y)X$, for $\omega\in\Lambda^1(M)$ and
for all $X,Y,Z \in T(M)$.

Computing the Gaussian curvature $K$, from the equation of
Gauss it follows
\begin{equation}
\label{dmn:K} K=\det A -\cos^2\theta.
\end{equation}
Now, taking into account the decomposition of any vector field $X\in T(M)$ as
$X=X_H+g(X,T)\dt$, the expression of $\dt$ from \eqref{dmn:t}
and the formulas {\bf (G)} and {\bf (W)}, the following proposition
holds as in \cite{dmn:DM09}.

{\bf Proposition A.} 
\it
Let $X$ be an arbitrary tangent vector to $M$. Then we have
\begin{equation}
\label{dmn:p1_1}
\nabla_X T =\cos \theta AX \\
\end{equation}
\begin{equation}
\label{dmn:p1_2} \hspace{10mm}X(\cos\theta)=-g(AX,T).
\end{equation}
\rm

The set of equations {\bf (E.G.)}, {\bf (E.C.)},
\eqref{dmn:p1_1} and \eqref{dmn:p1_2} are called the \emph{compatibility equations}.
The following result was given in \cite{dmn:Dan}:

{\bf Theorem B.} \it
Let $M$ be a simply connected Riemannian surface endowed with the metric $g$ and
its corresponding Levi-Civita connection $\nabla$. Let $A$ be a field of
symmetric operators $A_p:T_p(M)\rightarrow T_p(M)$ and $T$ a vector field on $M$
with $\|T\|^2=\sin^2\theta$, where $\theta$ is a smooth function defined on $M$.
Assume that $(g,A,T,\theta)$ satisfies the compatibility equations for $\H^2\times\R$.
Then there exists an isometric immersion $F:M\rightarrow \H^2\times\R$ such that
the shape operator with respect to the unit normal $\xi$ is given by $A$ and
$\partial_t=F_*T+\cos\theta \xi$. Moreover, the immersion is unique up to isometries of $\H^2\times\R$
preserving the orientation of both $\H^2$ and $\R$. \rm

\section{Surfaces in $\H^2\times\R$}

In this section we suppose that the angle function $\theta$ is different from $0$ and $\frac{\pi}{2}$.

\begin{proposition} 
\label{dmn:prop1}
If $\theta$ is never 0 or $\frac\pi2$, then we can choose local coordinates $(x,y)$  on the
surface $M$ isometrically immersed in $\Mt$ with $\partial_x$ in the direction of $T$ such that the metric on
$M$ has the following form
\begin{equation}
\label{dmn:p2_1} g= \frac{1}{\sin^2\theta}dx^2+\b ^2(x,y)dy^2.
\end{equation}
In the basis $\{\partial_x,\ \partial_y\}$ the shape operator $A$
can be expressed as
\begin{equation}
\label{dmn:p2_2}
A = \left( \begin{array}{ccc}
\theta_x\sin\theta & \theta_y\sin\theta \\
\frac{\theta_y}{\sin\theta\b^2}&\frac{\sin^2\theta\b_x}{\cos\theta\b}\end{array}
\right)
\end{equation}
and the functions $\theta$ and $\beta$ are related by the PDE
\begin{equation}
\label{dmn:p2_3} \frac{\sin^2\theta}{\cos\theta}\frac{\b_{xx}}{\b} +
\frac{\sin\theta\theta_x}{\cos^2\theta}\frac{\b_{x}}{\b}
+\frac{\theta_{y}}{\sin\theta}\frac{\b_y}{\b^3}+
\left(2\frac{\cos\theta\theta_y^2}{\sin^2\theta}-\frac{\theta_{yy}}{\sin\theta}\right)\frac{1}{\b^2}-\cos\theta=0.
\end{equation}
\end{proposition}
\begin{proof}
Here, and for the rest of the paper, we denote, for the sake of simplicity, $\frac{\partial}{\partial x}=\partial_x$ and
$\frac{\partial f}{\partial x}=f_x$ for any function $f$.

From the general theory, choosing an arbitrary point $p\in M$ such
that the angle function $\theta(p)\neq 0$ and
$\theta(p)\neq\frac{\pi}{2}$, we can take local orthogonal coordinates $(x,y)$
such that $\dx$ is in the direction of $T$ and the metric
$g$ on $M$ has the form
\begin{equation}
\label{dmn:pp2_1} g=\a^2(x,y)dx^2+\b^2(x,y)dy^2
\end{equation}
for certain functions $\a$ and $\b$ on $M$.

The Levi-Civita connection for this metric is given by the
following expressions
$$
\n_{\dx} \dx =\frac{\a_x}{\a}\dx - \frac{\a\a_y}{\b^2}\dy,\ \
\n_{\dx} \dy = \n_{\dy} \dx = \frac{\a_y}{\a}\dx + \frac{\b_x}{\b}\dy$$
$$
\n_{\dy} \dy = -\frac{\b\b_x}{\a^2}\dx + \frac{\b_y}{\b}\dy.
$$
In order to determine the shape operator $A$, we use formula
\eqref{dmn:p1_1} for $X=\dx$ and respectively for $X=\dy$. Since
$T=\frac{\sin\theta}{\a}\dx$ we get
\begin{eqnarray}
\label{dmn:pp2_5}
&&A\dx = \frac{\theta_x}{\a}\dx -\tan\theta\frac{\a_y}{\b^2}\dy
\\[2mm]
&&\label{dmn:pp2_6} A\dy= \frac{\theta_y}{\a}\dx + \tan\theta\frac{\b_x}{\a\b}\dy.
\end{eqnarray}

On the other hand, since $A$ is symmetric, i.e. $g(A\dy,\dx)=g(\dy,
A\dx)$ we obtain
\begin{equation}
\label{dmn:pp2_7} \tan\theta\a_y+\a\theta_y=0.
\end{equation}

It follows that the shape operator is given by
$
A = \left( \begin{array}{ccc}
\frac{\theta_x}{\a} & \frac{\theta_y}{\a}  \vspace{2mm} \\
\frac{\a\theta_y}{\b^2}& \frac{\tan\theta\b_x}{\a\b}\end{array}
\right).
$

After a first integration in \eqref{dmn:pp2_7} one obtains
$\displaystyle \a=\frac{\phi(x)}{\sin\theta}$, where $\phi$ is a function on $M$
depending on $x$. Changing the $x$-coordinate we can assume that
$\displaystyle \a=\frac{1}{\sin\theta}$ and substituting it in
\eqref{dmn:pp2_1} we get \eqref{dmn:p2_1}. Moreover, replacing
$\a$ in the previous expression of $A$ we obtain the shape operator
given exactly by formula \eqref{dmn:p2_2}.
In order to find a relation between $\b$ and $\theta$, we substitute in the
Codazzi equation {\bf (E.C.)} $X=\dx,\ Y=\dy$, $T=\sin^2\theta\dx$ and we get
\begin{equation*}
\n_{\dx} (A\dy) -\n_{\dy}(A\dx)-\cos\theta\dy=0.
\end{equation*}
After straightforward calculations, the PDE \eqref{dmn:p2_3} is obtained, concluding the proof.

\end{proof}
\begin{remark}\rm
Every two functions $\theta$ and $\beta$ defined on a smooth simply connected surface $M$, related by
\eqref{dmn:p2_3}, determine an isometric immersion of $M$ into ${\mathbb{H}}^2\times{\mathbb{R}}$
such that the shape operator is given by \eqref{dmn:p2_2}.
\end{remark}
\begin{proof}
Construct the metric $g$ as in \eqref{dmn:p2_1} and define the field of operators $A$ such that
its matrix in the local basis $\{\partial_x,\partial_y\}$ is given by \eqref{dmn:p2_2}.
It is easy to notice that all $A_p$ are symmetric ($p\in M$). Take $T=\sin^2\theta~\partial_x$.
At this moment Theorem B plays an essential role, and combining it with the previous Proposition~\ref{dmn:prop1}
we get the conclusion.

\end{proof}
Once the background of the study of a surface $M$ in $\H^2\times \R$ is established,
we are interested to find some particular classes
of surfaces involving the minimality and flatness conditions.
Following the same idea as in the general case but under the
restriction imposed by the minimality condition, one gets
\begin{proposition} 
\label{dmn:prop2}
If a surface $M$ isometrically immersed in $\H^2\times\R$ is
minimal and $\theta\neq0,\frac\pi 2$ then we can choose local coordinates $(x,y)$ such that
$\dx$ is in the direction of $T$, the metric is given by
\begin{equation}
\label{dmn:p3_1} g=\frac{1}{\sin^2\theta}~(dx^2 +dy^2)
\end{equation}
being conformal equivalent to the Euclidean metric
and the shape operator w.r.t the basis $\{\dx, \dy\}$ has the
expression
\begin{equation}
\label{dmn:p3_2}
A = \sin\theta\left(
\begin{array}{ccc}
\theta_x & \theta_y  \\[2mm]
\theta_y& -\theta_x
\end{array} \right).
\end{equation}
Moreover, denoting $\Delta=\sin^2\theta(\dx^2+\dy^2)$ the Laplacian
of the surface $M$, the angle function $\theta$ satisfies the PDE
\begin{equation}
\label{dmn:p3_3} \Delta
\ln\left(\tan\left(\frac{\theta}{2}\right)\right)=-\cos\theta.
\end{equation}
\end{proposition}
\begin{proof}
Using the results included in Proposition~\ref{dmn:prop1} and exploiting \eqref{dmn:p2_2}, the minimality
condition $H=0$ becomes
$
\cos\theta\theta_x\b+\sin\theta\b_x=0.
$
Integrating once w.r.t. $x$ one
gets $\displaystyle \b(x,y)=\frac{\psi(y)}{\sin\theta}$ for some
function $\psi$ on $M$ depending only on $y$. Making a change of the
$y$-coordinate one can assume
$\displaystyle\b(x,y)=\frac{1}{\sin\theta}$. Substituting it
in \eqref{dmn:p2_1} and \eqref{dmn:p2_2} the relations \eqref{dmn:p3_1} and \eqref{dmn:p3_2} are
proved. Note that $(x,y)$ are {\em isothermal coordinates}.

Computing the partial derivatives of $\b$ and replacing them in \eqref{dmn:p2_3}, an equivalent
condition is obtained
\begin{equation}
\label{dmn:pp3_1}
\cos\theta(\theta_x^2+\theta_y^2-1)-\sin\theta(\theta_{xx}+\theta_{yy})=0.
\end{equation}
By straightforward computations this is equivalent with \eqref{dmn:p3_3}.

\end{proof}
\begin{remark}\rm
Every smooth function $\theta$ defined on a smooth simply connected surface $M$
satisfying the elliptic equation \eqref{dmn:p3_3}, gives rise to an isometric minimal immersion of
$M$ into ${\mathbb{H}}^2\times{\mathbb{R}}$ such that the shape operator is given by \eqref{dmn:p3_2}.
\end{remark}

\begin{example}\rm
In order to give an example of angle function $\theta$ that corresponds
to a minimal surface in $\H^2\times\R$,
it is necessary to solve first the equation \eqref{dmn:pp3_1}. Let us
look for $\theta$ such that $\theta_x=k\theta_y$, for some non-vanishing real constant $k$.
Taking the derivatives with respect to $x$, respectively $y$, we get $\theta_{xx}=k\theta_{xy}$
and $\theta_{xy}=k\theta_{yy}$.
Substituting in \eqref{dmn:pp3_1} the expressions
$\theta_x^2+\theta_y^2=(k^2+1)\theta_y^2$ and
$\theta_{xx}+\theta_{yy}=\frac{k^2+1}{k}\theta_{xy}$ and integrating w.r.t. $x$ we get that
$
\frac{(k^2+1)\theta_y^2-1}{\sin^2\theta}=c(y).
$
It can be proved that $c(y)$ is constant, hence
$
\frac{\theta_y}{\sqrt{1+c\sin^2\theta}}=\pm \frac{1}{\sqrt{k^2+1}}
$
{\rm and}
$
\frac{\theta_x}{\sqrt{1+c\sin^2\theta}}=\pm \frac{k}{\sqrt{k^2+1}}\ .
$
Denoting $\ \displaystyle F(\theta\ | -c)=\int_0^\theta\frac{1}{\sqrt{1+c\sin^2\theta(t)}}\ dt$
the elliptic integral of first kind, it's differential is
$
 d\ F(\theta\ | -c)=\pm\ d\left(\frac{kx+y}{\sqrt{k^2+1}}\right)
$. It follows that $\theta=am\left(\pm\frac{kx+y}{\sqrt{k^2+1}}\ \big|\ -c\right)$, where $am$ denotes
the \emph{Jacobi amplitude}, namely the inverse of the elliptic function $F(\theta\ |\ -c)$.
\end{example}
If we look at the surfaces $M$ in $\H^2\times\R$ being at the same time
minimal and flat, we obtain the following classification theorem.
\begin{theorem}
The only surfaces in $\H^2\times\R$ which are both flat and minimal are given by
$f\times\R$, where $f$ is a geodesic line in $\H^2$.
\end{theorem}
\begin{proof}
If $\theta\neq 0$ or $\frac\pi 2$, then the metric takes the form \eqref{dmn:p3_1}
and the shape operator is given by \eqref{dmn:p3_2}. Asking for
the Gaussian curvature to vanish identically,
we obtain from \eqref{dmn:K} and using \eqref{dmn:p3_2} that
$\theta$ fulfills $\theta_x^2+\theta_y^2=-\cot^2\theta.$ Hence $\theta$ is constant different from $0$ or $\frac\pi 2$, which is
a contradiction. A surface with $\theta=0$ is not flat, hence $\theta=\frac\pi 2$ and the surface is of the form
$f\times\R$ with $f$ geodesic line in $\H^2$.

\end{proof}

\section{Surfaces in $\H^2\times\R$ with a canonical principal direction}

The problem of studying surfaces for which $T$ is a principal
direction arises from the previous papers on the constant angle
surfaces. We recall \cite{dmn:DM09} where it is proved that for a
constant angle surface $M$ in $\H^2\times\R$, i.e. for $\theta\in [0,\pi]$
constant, $T$ is a principal direction of the surface with
 the corresponding principal curvature $0$.
One natural generalization consists of the case when $T$ is assumed to be
a principal direction but the corresponding principal curvature
is different from $0$. At this point we denominate $T$ as a
{\em canonical principal direction}.

Let us say few words about the ambient space. There are many models describing the hyperbolic plane
and we will use the hyperboloid model, known also as the {\em Minkowski model}. Recall some basic facts
about this model which will be used in the sequel. We denote $\R_1^3$ the Minkowski $3$-space endowed with
the {\em Lorentzian metric}
$$
\<~,~\>=dx_1^2+dx_2^2-dx_3^2.
$$
For symmetry reasons, $\H^2$ can be modeled only by the upper sheet of the hyperboloid with two sheets,
namely,
$$
\H^2=\{(x_1,x_2,x_3)\in \R_1^3 ~|~ \ x_1^2+x_2^2-x_3^2=-1, ~ x_3>0 \}.
$$
The {\em Lorentzian cross-product} for two vectors $u,v\in\R_1^3$ is defined as
$\boxtimes:\R_1^3\times\R_1^3\rightarrow\R_1^3,$
\newline
$\big((u_1,u_2,u_3),\ (v_1,v_2,v_3)\big)\mapsto(u_2v_3-u_3v_2,\ u_3v_1-u_1v_3,\ u_2v_1-u_1v_2)$.

Concerning the curves in Minkowski 3-space, we recall that a regular curve $\gamma:I\rightarrow \R_1^3$ is called {\em spacelike} if
$\<\dot{\gamma},\dot{\gamma}\> > 0$ everywhere, {\em timelike} if
$\<\dot{\gamma},\dot{\gamma}\> <0$ in any point, and respectively, {\em lightlike} if $\<\dot{\gamma},\dot{\gamma}\> =0$ everywhere.

In order to study under which conditions $T$ is a canonical
principal direction, we regard the surface $M$ as a surface immersed
in $\R_1^3\times\R$ (also denoted $\R_1^4$) having codimension $2$.

The metric on the ambient space is given by
$\widetilde{g}=dx_1^2+dx_2^2-dx_3^2+dt^2$.

At this point, let us consider the surface $M$ given by the
immersion $F:M\rightarrow \R_1^3\times\R$, $F=(F_1,\ F_2,\ F_3,\
F_4)$. Denote by $\widetilde{\xi}=(F_1,\ F_2,\ F_3,\ 0)$ the normal to
$\H^2\times\R$ in the points of $M$ and by $\xi=(\xi_1,\ \xi_2,\ \xi_3,\ \cos\theta)$ the normal to $M$ in
$\H^2\times\R$. (See for details \cite{dmn:DM09}.) Moreover, from now on, by $D^\perp$
we mean the normal connection of $M$ in $\R_1^4$ and by $R^\perp$ the normal curvature
$$
R^\perp(X,\ Y)=[D_{X}^\perp,\ D_{Y}^\perp ]- D^\perp_{[X,\ Y]}\ ,{\rm\ for\ all}\ X,Y\ {\rm tangent\ to}\ M.
$$

The first result in this section is the characterization
theorem of surfaces $M$ having $T$ as principal direction in terms of
the normal curvature.

\begin{theorem}[{\bf Characterization Theorem}]
\label{dmn:thm2}
Let $M$ be a surface isometrically immersed in $\H^2\times\R$ such that $\theta\neq 0$. $T$
is a principal direction if and only if $M$ is normally
flat in $\R_1^3\times\R$.
\end{theorem}
\begin{proof}
With the above considerations, for any $X$ tangent to $M$, we compute
$D^\perp_X\tilde{\xi}=-\cos\theta\<X,T\>\xi$
which implies
$D^\perp_X\xi=\cos\theta\<X,T\>\tilde{\xi}$.

Since Proposition~\ref{dmn:prop1} holds, the metric is given by
\eqref{dmn:p2_1}, and using the previous expressions, one has
$$
R^\perp(\dx,\ \dy)\xi=\sin\theta\theta_y\tilde{\xi}\qquad {\rm and}\qquad
R^\perp(\dx,\ \dy)\tilde{\xi}=-\sin\theta\theta_y \xi.
$$
Taking into account that $\xi$ and $\tilde{\xi}$ are unitary and $\sin\theta$ cannot vanish,
we get that $M$ is normally flat if and only if $\theta_y=0$.
On the other hand, $T$ is a canonical principal direction if and only if $\theta_y=0$.
This follows from expression \eqref{dmn:p2_2} of the Weingarten operator $A$.
Hence, we get the conclusion.
\end{proof}
An analogue result with Proposition~\ref{dmn:prop1}, formulated for surfaces in
$\H^2\times\R$ having $T$ as principal direction, is the following
\begin{proposition}
\label{dmn:prop3}                      
Let $M$ be isometrically immersed in $\H^2\times\R$ such that $\theta\neq 0, \frac\pi 2$ with $T$ a
principal direction. Then, we can choose local coordinates $(x, y)$ such that $\dx$ is in the direction of $T$, the metric is given by
\begin{equation}
\label{dmn:p4_1} g=dx^2+\b^2(x,\ y)dy^2
\end{equation}
and the shape operator w.r.t. $\{\dx,\ \dy\}$ can be written as
\begin{equation}
\label{dmn:p4_2}
A = \left( \begin{array}{ccc}
\theta_x & 0 \\
0&\tan\theta\frac{\b_x}{\b}
\end{array} \right).
\end{equation}
Moreover, the functions $\theta$ and $\beta$ are related by the PDE
\begin{equation}
\label{dmn:p4_3} \b_{xx} + \tan\theta\theta_x\b_{x} -\b\cos^2\theta=0
\end{equation}
and $\theta_y=0$.
\end{proposition}
\begin{proof}
Looking back at the proof of Proposition~\ref{dmn:prop1}, if $T$ is a
principal direction, then one can choose local coordinates $(x,y)$
such that
\linebreak
$g=\a^2(x,y)dx^2+\b^2(x,y)dy^2$ and $\theta_y=0$, i.e.
the angle function $\theta(x, y)$ becomes function only of $x$.
This means that we can do a change of the $x$-coordinate such that
the metric is given now by \eqref{dmn:p4_1}. Moreover, following the
same line of proof as in Proposition~\ref{dmn:prop1}
 for $\alpha=1$ and $\theta_y=0$, we obtain that the shape operator is given by \eqref{dmn:p4_2}.
Finally, the equation of Codazzi yields \eqref{dmn:p4_3}.
\end{proof}
\begin{remark} \rm
For every two functions $\theta$ and $\beta$ defined on a smooth simply connected surface $M$ such that
$\beta_{xx}+ \tan\theta\theta_x\beta_{x} -\beta\cos^2\theta=0$ and $\theta_y=0$
for certain coordinates
$(x,y)$, we can construct an isometric immersion $F:M\longrightarrow {\mathbb{H}}^2\times{\mathbb{R}}$
with the shape operator \eqref{dmn:p4_2} and such that it has a canonical principal direction.
\end{remark}
\begin{remark}\label{dmn:remark4}\rm
 Let $M$ be an isometrically immersed surface in $\H^2\times\R$ such that
$T$ is a principal direction. Coordinates $(x,y)$ on $M$ such that the hypotheses of Proposition~\ref{dmn:prop3} are fulfilled,
i.e. $\partial_x$ is collinear with $T$ and the metric $g$ has the form $g=dx^2+\b^2(x,y)dy^2$, will
be called {\em canonical coordinates}. Of course, they are not unique. More precisely,
if $(x,y)$ and $(\overline{x},\overline{y})$ are both canonical coordinates, then
they are related by $\overline{x}=\pm\ x+c$ and $\overline{y}=\overline{y}(y)$, where $c\in\R$.
\end{remark}

We are interested in solving equation \eqref{dmn:p4_3} in order to find explicit parametrizations of
surfaces $M$ in $\H^2\times\R$ with $T$ as principal direction.

We give first some additional results.

\begin{lemma} 
\label{dmn:lemma1}
Let us consider the following PDE {\rm (}in $\f${\rm)}
\begin{equation*}
\frac{\f_x^2}{\cos^2\theta(x)}-\f^2=\tilde{\mu}(y)
\end{equation*}
 with $\f$ a function of two variables $x$ and $y$, $\theta$ a function of one variable $x$ and
$\tilde{\mu}$ a function of $y$ that does not change sign. The set of solutions is
\begin{eqnarray}
\label{dmn:l1_3}
\f = \mu(y)\sinh(\phi(x)+\psi(y)) &  {\rm when } & \tilde{\mu}(y)=\mu^2(y) \\
\label{dmn:l1_4}
\f = \mu(y)\cosh(\phi(x)+\psi(y)) & {\rm when } & \tilde{\mu}(y)=-\mu^2(y)\\
\label{dmn:l1_2}
\f = \psi(y)e^{\pm\phi(x)} \hspace{20mm} & {\rm when } & \tilde{\mu}(y)=0
\end{eqnarray}
where $\psi$ is an integration function depending only on $y$, while $\phi$ depends only on $x$
and denotes a primitive of $\cos\theta$.
\end{lemma}

\begin{lemma}
\label{dmn:lemma2} 
The ODE {\rm(}in $\f${\rm)}
\begin{equation*}
\f''+\tan\theta~\theta'\f'-\cos^2\theta\f=0
\end{equation*}
has the solution
\begin{equation*}
\f=c_1\sinh\phi(x)+c_2\cosh\phi(x)
\end{equation*}
where $\f$ and $\theta$ are one variable functions of $x$, $\phi$ is a primitive of $\cos\theta$ and $c_1,\ c_2$  are real constants.
\end{lemma}


We can now state the following theorem.

\begin{theorem} 
\label{dmn:thm3}
If $F:M\rightarrow\H^2\times\R$ is an isometric immersion with $\theta\neq0, \frac\pi 2$, then $T$
is a principal direction if and only if $F$ is given,
up to isometries of $\H^2\times\R$, by
$$
F(x,\ y)=(F_1(x,\ y),\ F_2(x,\ y),\ F_3(x,\ y),\ F_4(x))
$$
with  $F_j(x,y)=A_j(y)\sinh\phi(x)+B_j(y)\cosh\phi(x)$, for $j=1,2,3$ and
$F_4(x)=\displaystyle\int_0^x\sin\theta(\tau)d\tau$, where $\phi'(x)=\cos\theta$.
The six functions $A_j$ and $B_j$ are found in one of the following cases
\begin{itemize}
\item {\bf Case ${\mathbf 1.}$} 
\begin{eqnarray*}
A_j(y)=\int_0^y H_j(\tau)\cosh\psi(\tau)d\tau +c_{1j}\hspace{9mm} \\
B_j(y)=\int_0^y H_j(\tau)\sinh\psi(\tau)d\tau +c_{2j}\hspace{10mm} \\
H_j'(y)=B_j(y)\sinh\psi(y)-A_j(y)\cosh\psi(y)
\end{eqnarray*}
\item {\bf Case ${\mathbf 2.}$} 
\begin{eqnarray*}
A_j(y)=\int_0^y H_j(\tau)\sinh\psi(\tau)d\tau +c_{1j}\hspace{12mm}\\
B_j(y)=\int_0^y H_j(\tau)\cosh\psi(\tau)d\tau +c_{2j}\hspace{12mm}\\
H_j'(y)=-A_j(y)\sinh\psi(y)+B_j(y)\cosh\psi(y)
\end{eqnarray*}
\item {\bf Case ${\mathbf 3.}$} 
\begin{eqnarray*}
A_j(y)=\pm \int_0^y H_j(\tau)d\tau +c_{1j}\hspace{28mm}\\
B_j(y)=\int_0^y H_j(\tau)d\tau +c_{2j}\hspace{32mm}\\
H_j'(y)=c_{2j}\mp c_{1j}\hspace{15mm}\hspace{32mm}
\end{eqnarray*}
\end{itemize}
where $H=(H_1,H_2,H_3)$ is a curve on the de Sitter space $\S_1^2$,
$\psi$ is a smooth function on $M$ and $c_1=(c_{11},c_{12},c_{13})$,
$c_2=(c_{21},c_{22},c_{23})$ are constant vectors.
\end{theorem}
\begin{proof}
We choose canonical coordinates $(x,y)$ as in Proposition~\ref{dmn:prop3}.

First, let us determine the $4^{{\rm th}}$ component of $F$.
Taking the derivatives w.r.t. $x$ and respectively w.r.t. $y$, one has
$\displaystyle (F_4)_x=\widetilde g(F_x,\dt)=\sin\theta$, $(F_4)_y=0$. It
follows that $F_4$ depends only on $x$ and it can be expressed,
after a translation of $M$ along the $t$-axis, as $F_4(x)=\displaystyle\int_0^x\sin\theta(\tau)d\tau$.

We point our attention now on the first three components. We are able to write the Levi-Civita connection of the metric
given by \eqref{dmn:p4_1}
\begin{eqnarray}
\label{dmn:tt4_1}
&&\n_{\dx} \dx =0\\[2mm]
\label{dmn:tt4_2}
&&\n_{\dx} \dy = \n_{\dy} \dx = \frac{\b_x}{\b}\dy\\[2mm]
\label{dmn:tt4_3}
&&\n_{\dy} \dy =-\b\b_x\dx + \frac{\b_y}{\b}\dy.
\end{eqnarray}

Recall that $M$ is a codimension $2$ surface in $\R_1^4$. Computing explicitly the two normals, we have
\begin{equation}
\label{dmn:tt4_4} \xi=(\xi_1,\ \xi_2,\ \xi_3,\ \cos\theta),\ {\rm where}\
\xi_{j}=-\tan\theta(F_j)_x,\ j=\overline{1,3}
\end{equation}
\begin{equation}
\label{dmn:tt4_5} \tilde{\xi}=(F_1,\ F_2,\ F_3,\ 0)
\end{equation}
with corresponding shape operators $A$ given by \eqref{dmn:p4_2} and
\begin{equation}
\label{dmn:tt4_6} \widetilde{A} = \left( \begin{array}{ccc}
-\cos^2\theta & 0 \\
0&-1
\end{array} \right).
\end{equation}
Cf. also \cite{dmn:DM09}.

Using the Gauss formula {\bf (G)} in combination with \eqref{dmn:tt4_1}-\eqref{dmn:tt4_5} one gets
\begin{equation}
\label{dmn:tt4_7} (F_j)_{xx}=\cos^2\theta F_j
-\tan\theta\theta_x(F_j)_x,\  j=\overline{1,3} \hspace{18mm}
\end{equation}
\begin{equation}
\label{dmn:tt4_8} (F_j)_{xy}=\frac{\b_x}{\b}(F_j)_y,\  j=\overline{1,3} \hspace{43mm}
\end{equation}
\begin{equation}
\label{dmn:tt4_9}
(F_j)_{yy}=\b^2F_j-\frac{1}{\cos^2\theta}\b\b_x(F_j)_x+\frac{\b_y}{\b}(F_j)_y,\  j=\overline{1,3}.
\end{equation}
We remark that \eqref{dmn:tt4_7} does not depend on $\b$.

Our aim is to determine the function $\b$. A first integration in \eqref{dmn:p4_3} leads to
$\displaystyle \frac{\b_x^2}{\cos^2\theta}-\b^2=\tilde{k}(y)
$
where $\tilde k$ is an arbitrary function on $M$ depending only on $y$.
Solving this PDE according to Lemma~$\ref{dmn:lemma1}$ we distinguish the
following cases:

\smallskip

{\bf Case ${\mathbf 1.}$} If $\tilde{k}(y)=\mu^2(y)$, the solution has
the form \eqref{dmn:l1_3}.

Changing the $y$-coordinate, we consider $\b=\sinh(\phi(x)+\psi(y))$,
with $\phi'(x)=\cos\theta$ and hence the metric
has the form
\begin{equation}
\label{dmn:tt4_17}
g=dx^2+\sinh^2(\phi(x)+\psi(y))dy^2.
\end{equation}

From \eqref{dmn:tt4_8} and \eqref{dmn:tt4_9} and taking into account
\eqref{dmn:tt4_17} we get
\begin{equation}
\label{dmn:tt4_18}
(F_j)_{xy}=\coth(\phi(x)+\psi(y))\cos\theta(F_j)_y
\end{equation}
respectively
\begin{equation}
\label{dmn:tt4_19}
(F_j)_{yy} = \sinh^2(\phi(x)+\psi(y))F_j-\frac{1}{2\cos\theta}\sinh2(\phi(x)+\psi(y))(F_j)_x +
\end{equation}
$$
     + \coth(\phi(x)+\psi(y))\psi'(y)(F_j)_y.\qquad
$$
After two consecutive integrations in \eqref{dmn:tt4_18}  w.r.t. $x$ and w.r.t. $y$
we have
$$
F_j(x,\ y)=\int\limits_0\limits^y
H_j(\tau)\sinh(\phi(x)+\psi(\tau))d\tau+I_j(x).
$$
In addition, since $F_j$ fulfills \eqref{dmn:tt4_7}, each $I_j$
satisfies the equation
\linebreak
$
I_j''+\tan\theta\theta_x I_j'-\cos^2\theta I_j=0
$
which, by Lemma~$\ref{dmn:lemma2}$, has the general solution
$I_j=c_{1j}\sinh\phi(x)+c_{2j}\cosh\phi(x)$. Hence we find
\begin{equation}
\label{dmn:tt4_2_20}
F_j(x,\ y)=\int\limits_0\limits^y H_j(\tau)\sinh(\phi(x)+\psi(\tau))d\tau+\qquad
\end{equation}
$$
\qquad  + c_{1j}\sinh\phi(x)+c_{2j}\cosh\phi(x).
$$
We still have to use the condition \eqref{dmn:tt4_19}. This yields
\begin{equation}
\begin{array}{l}
\label{dmn:tt4_2_21}
H_j'(y)=-\cosh(\phi(x)+\psi(y))\displaystyle\int\limits_0^y H_j(\tau)\cosh(\phi(x)+\psi(\tau))d\tau+\\
\qquad\qquad + \sinh(\phi(x)+\psi(y))\displaystyle\int\limits_0^y H_j(\tau)\sinh(\phi(x)+\psi(\tau))d\tau+\\
\qquad\qquad + c_{2j}\sinh\psi(y)-c_{1j}\cosh\psi(y).
\end{array}
\end{equation}

We remark that apparently the right hand in \eqref{dmn:tt4_2_21} depends on both $x$ and $y$,
while the left one depends only on $y$.
Further-on we will see that everything depends only of $y$.
If we define
\begin{equation}
\label{dmn:tt4_2_22}
\begin{array}{l}
A_j(y)=\displaystyle\int\limits_0^y H_j(\tau)\cosh\psi(\tau)d\tau + c_{1j}\\
B_j(y)=\displaystyle\int\limits_0^y H_j(\tau)\sinh\psi(\tau)d\tau + c_{2j}
\end{array}
\end{equation}
then we can rewrite \eqref{dmn:tt4_2_20} and  \eqref{dmn:tt4_2_21} in a simpler form
\begin{equation}
\label{dmn:tt4_2_24} F_j(x,\ y)=A_j(y)\sinh\phi(x)+B_j(y)\cosh\phi(x)
\end{equation}
respectively
\begin{equation*}
\begin{array}{rcl}
\label{dmn:tt4_2_25} H_j'(y)=B_j(y)\sinh\psi(y)-A_j(y)\cosh\psi(y).
\end{array}
\end{equation*}
We will keep in mind that $A_j$ and $B_j$ for $j=\overline{1,3}$ depend on $y$ but, for simplicity,
we drop the "$y$" in writing.

Let us consider $\epsilon_1=\epsilon_2=1$ and $\epsilon_3=-1$. Then we have
\begin{equation}
\label{dmn:tt4_2_26}
\hspace{-30mm}\sum\limits_{j=1}\limits^3\epsilon_jF_j^2=-1
\end{equation}\vspace{-4mm}
\begin{equation}
\label{dmn:tt4_2_27}
\hspace{-15mm}\sum\limits_{j=1}\limits^3\epsilon_j{(F_j)}_x^2=\cos^2\theta(x)
\end{equation}\vspace{-4mm}
\begin{equation}
\label{dmn:tt4_2_28}
\hspace{-20mm}\sum\limits_{j=1}\limits^3\epsilon_j{(F_j)}_x{(F_j)}_y=0
\end{equation}\vspace{-4mm}
\begin{equation}
\label{dmn:tt4_2_29}
\sum\limits_{j=1}\limits^3\epsilon_j{(F_j)}_y^2=\sinh^2(\phi(x)+\psi(y)).
\end{equation}
Combining in a  proper manner \eqref{dmn:tt4_2_24} with \eqref{dmn:tt4_2_26}--\eqref{dmn:tt4_2_28},
denoting
\linebreak
$A=(A_1,\ A_2,\ A_3)$ and $B=(B_1,\ B_2,\ B_3)$ we get the following relations
written in terms of the Lorentzian scalar product
\begin{equation}
\label{dmn:tt4_2_35}
\<A,\ A\>=1,  \<B,\ B\>=-1,  \<A,\ B\>=0,  \<A',\ B\>=\<A,\ B'\>=0. 
\end{equation}
As consequence one has
$$
\<A,\ A'\>=0,\ \<B,\ B'\>=0,\ \<A,\ H\>=0,\ \<B,\ H\>=0.
$$
Finally, developing \eqref{dmn:tt4_2_29} one obtains
\begin{equation*}
\label{dmn:tt4_2_36}
\<H,\ H\>=\<A',\ A'\>-\<B',\ B'\>=1.
\end{equation*}
Moreover $\<H',\ H'\>=1$.
We conclude that $H$ is a unit speed spacelike curve on the Lorentzian unit sphere, known as the de Sitter space $\S_1^2$.

\medskip

{\bf Case ${\mathbf 2.}$} If $\tilde{k}(y)=-\mu^2(y)$, the solution
has the form \eqref{dmn:l1_4}.

Again, changing the $y$-coordinate we consider $\b=\cosh(\phi(x)+\psi(y))$,
where $\phi'(x)=\cos\theta$ and the metric in this case
is given by
\begin{equation*}
 g=dx^2+\cosh^2(\phi(x)+\psi(y))dy^2.
\end{equation*}
In a similar way as in {\bf Case $\mathbf{1}$}, by straightforward computations one gets
\begin{equation*}
\label{dmn:tt4_2_38}
F_j(x,\ y)=\int\limits_0\limits^y H_j(\tau)\cosh(\phi(x)+\psi(\tau))d\tau+c_{1j}\sinh\phi(x)+c_{2j}\cosh\phi(x)
\end{equation*}
with
\begin{equation*}
\label{dmn:tt4_2_39}
\begin{array}{rcl}
H_j'(y)&=&-\sinh(\phi(x)+\psi(y))\displaystyle\int\limits_0^y H_j(\tau)\sinh(\phi(x)+\psi(\tau))d\tau+\\
&&\cosh(\phi(x)+\psi(y))\displaystyle\int\limits_0^y H_j(\tau)\cosh(\phi(x)+\psi(\tau))d\tau+\\
&&c_{2j}\cosh\psi(y)-c_{1j}\sinh\psi(y).
\end{array}
\end{equation*}
As in the previous case, we define the following quantities
\begin{equation*}
\label{dmn:tt4_2_40}
\begin{array}{l}
A_j(y)=\displaystyle\int\limits_0\limits^y H_j(\tau)\sinh\psi(\tau)d\tau + c_{1j}\quad {\rm and}\\[2mm]
B_j(y)=\displaystyle\int\limits_0\limits^y H_j(\tau)\cosh\psi(\tau)d\tau + c_{2j}.
\end{array}
\end{equation*}
Then we have
\begin{equation*}
\label{dmn:tt4_2_42}
F_j(x,\ y)=A_j(y)\sinh\phi(x)+B_j(y)\cosh\phi(x)
\end{equation*}
\begin{equation*}
\begin{array}{rcl}
\label{dmn:tt4_2_43}
H_j'(y)=-A_j(y)\sinh\psi(y)+B_j(y)\cosh\psi(y).
\end{array}
\end{equation*}
Formulas \eqref{dmn:tt4_2_26}-\eqref{dmn:tt4_2_28} together with
\begin{equation*}
\label{dmn:tt4_2_44}
\sum\limits_{j=1}\limits^3\epsilon_j{(F_j)}_y^2=\cosh^2(\phi(x)+\psi(y))
\end{equation*}
imply that the expressions \eqref{dmn:tt4_2_35} again hold, and that
\begin{equation*}
\label{dmn:tt4_2_45}
\<H,\ H\>=-\<A',\ A'\>+\<B',\ B'\>=1.
\end{equation*}
In this case we find $\<H',\ H'\>=-1$, hence $H$ is an unit speed timelike curve in the de Sitter space $\S_1^2$.

\medskip

{\bf Case ${\mathbf 3.}$} $\tilde{k}(y)=0$. The solution is given by \eqref{dmn:l1_2}.

After a change of the $y$-coordinate, locally  we have $\b=e^{\pm\phi(x)}$,
where $\phi'(x)=\cos\theta$, and the metric is given by
\begin{equation*}
\label{dmn:tt4_3_11} 
g=dx^2+e^{\pm 2\phi(x)}dy^2.
\end{equation*}
By straightforward computations we find in this case
\begin{equation*}
\label{dmn:tt4_3_14}
 F_j(x,y)=A_j(y)\sinh\phi(x)+B_j(y)\cosh\phi(x)
\end{equation*}
where we denoted
\begin{equation}
\label{dmn:tt4_3_15}
A_j=\pm \int\limits_0\limits^y H_j(\tau)d\tau+c_{1j}\quad
{\rm and}\quad  B_j=\int\limits_0\limits^y H_j(\tau)d\tau+c_{2j}
\end{equation}
with
\begin{equation}
\label{dmn:tt4_3_16}
H_j'(y)=c_{2j}\mp c_{1j}.
\end{equation}
Applying the same technique as in previous cases, we get
\begin{eqnarray}
\label{dmn:tt4_3_17}
\<A,\ A\>=1,\ \<B,\ B\>=-1,\
\<A,\ B\>=0,\ \<A',\ B\>=0,\ \<A,\ B'\>=0.
\nonumber
\end{eqnarray}
Moreover $\<A,\ H\>=0,\ \<B,\ H\>=0$, since $A'=\pm H$, $B'=H$.

Remark that $H$ is unitary, i.e. $\<H,\ H\>=1$, with $H'\neq0$.

Denoting by $c_1=(c_{11},\ c_{12},\ c_{13})$, $c_2=(c_{21},\ c_{22},\ c_{23})$,
from \eqref{dmn:tt4_3_16} we compute
$
H=(c_2\mp c_1)y+c_3$ where $c_3=(c_{31}, c_{32}, c_{33})$ is a constant vector.
Plugging this value in \eqref{dmn:tt4_3_15} we get
$$
A=\frac{\pm c_2-c_1}{2}y^2\pm c_3y+c_1\qquad {\rm and }\qquad
B=\frac{c_2\mp c_1}{2}y^2+c_3y+c_2.
$$
We conclude with the following relations satisfied by $c_1$, $c_2$ and $c_3$
\begin{eqnarray}
\label{dmn:tt4_3_19}
\<c_1,\ c_1\>=1,\ \<c_2,\ c_2\>=-1,\ \<c_3,\ c_3\>=1,\\ \nonumber
\<c_1,\ c_2\>=0,\ \<c_1,\ c_3\>=0,\ \<c_2,\ c_3\>=0.\hspace{2mm}
\end{eqnarray}
From \eqref{dmn:tt4_3_16} and \eqref{dmn:tt4_3_19} it follows that $\<H',\ H'\>=0$.
Hence, in this case, $H$ is a lightlike curve in the
de Sitter space $\S_1^2$.

{\bf Conversely}, in each of the three cases we
prove that the corresponding surface has $T$ as principal direction.
Motivated by the fact that the idea of the proof is the same in
all cases, we sketch the proof only in the first of them. \\
We prove that the surface parametrized by
$$
\displaystyle F(x,y)=\left(F_1(x,y),\ F_2(x,y),\ F_3(x,y),\ \int_{0}^{x}\sin\theta(\tau)d\tau\right)
$$
where $F_j(x,y)=A_j(y)\sinh\phi(x)+B_j(y)\cosh\phi(x)$ with $A_j$ and $B_j$ given by
\eqref{dmn:tt4_2_22} has $T$ as principal direction.

Since
$\widetilde g(F_x, F_x)=1$, $\widetilde g(F_x, F_y)=0$, $\widetilde g(F_y, F_y)=\sinh^2(\phi(x)+\psi(y))$,
it follows that the metric $g$ can be written in form \eqref{dmn:p4_1}.
Computing the shape operator (e.g. from \eqref{dmn:pp2_5} and \eqref{dmn:pp2_6})
and using its symmetry we get $\theta_y=0$ and \eqref{dmn:p4_2}.
It is easy to prove that $\widetilde g(F_x, T)=\sin\theta$ and $\widetilde g(F_y,T)=0$
concluding that $T$ is a principal direction of the surface $M$ parametrized by $F$.

At this moment, the theorem is completely proved.
\end{proof}
\begin{remark}
\rm
Note that in order to reach an unified description of surfaces for which $T$ is a principal direction, we obtain
that the immersion $F:M\rightarrow\H^2\times\R$ is given by
$$
F(x,y)=\left(A(y)\sinh\phi(x)+B(y)\cosh\phi(x),\ \int_0^x\sin\theta(\tau)d\tau\right).
$$
In each case of the theorem of classification $A$ is a curve in $\S^2_1$ and
$B$ is a curve in $\H^2$ orthogonal to $A$ and such that the two speeds $A'$ and $B'$ are parallel.
Denoting by $H$ the unit vector of their common direction, one has $H=A\boxtimes B$ and moreover
\begin{itemize}
\item $H$ is a spacelike curve in the first case
\item $H$ is a timelike curve in the second case
\item $H$ is a lightlike curve in the last case.
\end{itemize}
\end{remark}


\begin{theorem}[\bf{Classification theorem}]
\label{dmn:thm4}
If $F:M\rightarrow \H^2\times\R$ is an isometric immersion with angle function $\theta\neq0, \frac\pi 2$, then $T$ is a principal direction
if and only if $F$ is given locally, up to isometries of the ambient space by
$$
F(x,y)=\left(A(y)\sinh\phi(x)+B(y)\cosh\phi(x), \chi(x)\right)
$$
where $A(y)$ is a curve in $\S_1^2$, $B(y)$ is a curve in $\H^2$, such that
\linebreak
$\langle A,B\rangle=0$, $A'|| B'$ and where $(\phi(x),\chi(x))$ is a regular curve in $\R^2$.
The angle function $\theta$ of $M$ depends only on $x$ and coincides with the angle function of the curve $(\phi,\chi)$.
In particular we can arc length reparametrize $(\phi,\chi)$; then $(x,y)$ are canonical coordinates and $\theta'(x)=\kappa(x)$, the curvature of $(\phi,\chi)$.
\end{theorem}

\begin{remark} \rm
\label{rmk:changeAB}
Since $\phi$ is determined up to constants ($\phi'(x)=\cos\theta$) we ask ourselves in which way a change
$\widetilde{\phi}(x)=\phi(x)-\phi_0$, for certain $\phi_0\in \R$, affects the parametrization of the surface.
If we take
$$
F(x,y)=\left(\widetilde{A}(y)\sinh\tilde\phi(x)+\widetilde{B}(y)\cosh\tilde\phi(x),\ \chi(x)\right)
$$
we immediately obtain
$\widetilde{A}(y)=A(y)\cosh\phi_0+B(y)\sinh\phi_0$ and
\linebreak
$\widetilde{B}(y)=A(y)\sinh\phi_0+B(y)\cosh\phi_0$.
The new curves $\widetilde{A}$ and $\widetilde{B}$ satisfy the same conditions as the initial curves $A$ and $B$.
This kind of change could be useful in some proofs (see e.g. Theorem~\ref{dmn:thm5}).
\end{remark}


Now, we would like to give some examples of surfaces
that can be retrieved from Theorem~\ref{dmn:thm3}.
In order to do this, we have to determine for each of the three cases the functions
$A_j$, $B_j$, $j=\overline{1,3}$ which satisfy the required conditions.

\smallskip
Let us consider first the easiest situation when $\psi(y)=0$ for all $y$.
\begin{example}\rm
\label{ex:2}
In {\bf Case 1} of Theorem~\ref{dmn:thm3} we get
$$
\displaystyle
A_j(y)=\int_0^y H_j(\tau)d\tau +c_{1j},
\
B_j(y)=c_{2j},
\
\displaystyle H_j'(y)=-\int_0^y H_j(\tau)d\tau -c_{1j}.
$$
\end{example}

It follows that $H(y)=l\cos y +m\sin y$, where $l,m\in \R^3$ are constant vectors.
Since $H$ is unitary, we find $\langle l,l\rangle=1$, $\langle m,m\rangle=1$,
$\langle l,m\rangle=0$.
Let us choose $l=(1,0,0)$ and $m=(0,-1,0)$. Going back to the expression of
$H'$ we obtain $c_{1}=-m$. Straightforward computations yield that
$A(y)=(\sin y, \cos y, 0)$ and we can choose $B(y)=(0,0,1)$.

The parametrization $F$ in this case is given by
\begin{equation*}
\begin{array}{l}
F(x,y) = \displaystyle \left(\sin y \sinh\Big(\int_0^x\cos\theta(\tau)d\tau\Big),\
\cos y \sinh\Big(\int_0^x\cos\theta(\tau)d\tau\Big),\ \right. \vspace{2mm} \\
\qquad\qquad\qquad \displaystyle\left. \cosh\Big(\int_0^x\cos\theta(\tau)d\tau\Big),\
\int_0^x\sin\theta(\tau)d\tau \right).
\end{array}
\end{equation*}
This surface can be obtained by taking the curve
$$
\left(\sinh\big(\int_0^x\cos\theta(\tau)d\tau\big),\
\cosh\big(\int_0^x\cos\theta(\tau)d\tau\big),\
\int_0^x\sin\theta(\tau)d\tau \right)
$$
in $\H^1\times {\mathbb{R}}\subset \H^2\times {\mathbb{R}}$ and rotating it
in an appropriate way (in the $(x_1x_2)$-plane of $\R^3_1\times\R$).

For example, if $\theta(x)=x$, which yields $\kappa=1$ in the classification theorem,
a nice parametrization arises, namely
\begin{equation*}
\begin{array}{l}
F(x,y) = \displaystyle \big(\sin y \sinh(\sin x),\
\cos y \sinh(\sin x),\
\displaystyle \cosh(\sin x),\ 1-\cos x \big).
\end{array}
\end{equation*}

\begin{example}
\rm
Let's see what parametrization springs up in {\bf Case 2} of Theorem~\ref{dmn:thm3}, when the function $\psi$
vanishes identically. Following the idea of the previous example,
we get $B(y)=(\sinh y,0, \cosh y)$ and we can take $A(y)=(0,1,0)$.
Hence, one obtains the following parametrization
\begin{equation*}
\begin{array}{lll}
F(x,y)& = &\displaystyle \left(\sinh y \cosh\Big(\int_0^x\cos\theta(\tau)d\tau\Big),\
\sinh\Big(\int_0^x\cos\theta(\tau)d\tau\Big),\ \right.\vspace{2mm} \\
& & \qquad \displaystyle\left. \cosh y \cosh\Big(\int_0^x\cos\theta(\tau)d\tau\Big),\
\int_0^x\sin\theta(\tau)d\tau \right).
\end{array}
\end{equation*}
For instance, if $\theta(x)=\arccos(x)$ the surface is given by

{\small
$F(x,y) = \left(\sinh y \cosh \frac{x^2}{2},\
\sinh \frac{x^2}{2},\ \cosh y \cosh \frac{x^2}{2},\
\frac{1}{2}\big(x \sqrt{1-x^2}+\arcsin x\big) \right).
$}
\end{example}

Some other interesting examples can be obtained considering $\psi(y)=y$.

\begin{example}\rm
For {\bf Case 1} we find the parametrization
$$\begin{array}{l}
F(x,y)=\Big(A(y)\sinh\big(\int_0^x\cos\theta(\tau)d\tau\big)+B(y)\cosh\big(\int_0^x\cos\theta(\tau)d\tau\big),
\\
\hspace{40mm}
\int_0^x\sin\theta(\tau)d\tau \Big)
\end{array}
$$
where
{\small
$$
A(y)=\Big(y\sinh y-\cosh y,\ -\frac{y^2}{2}\sinh y + y \cosh y,\ \frac{y^2}{2}\sinh y - y \cosh y+\sinh y\Big)
$$}
{\small
$$
B(y)=\Big(y\cosh y-\sinh y,\ -\frac{y^2}{2}\cosh y + y \sinh y,\ \frac{y^2}{2}\cosh y - y \sinh y+\cosh y\Big).
$$
}
\end{example}
\begin{example}\rm
For the second case of Theorem~\ref{dmn:thm3} we get a similar parametrization
$$
F(x,y)=\left(A(y)\sinh\Big(\int_0^x\cos\theta(\tau)d\tau\Big)+B(y)\cosh\Big(\int_0^x\cos\theta(\tau)d\tau\Big),\right.
$$
$$
\left.\int_0^x\sin\theta(\tau)d\tau \right)
$$
with
\begin{equation*}
\begin{array}{lll}
A(y)&=&\Big(\sinh(y\sqrt2)\sinh y-\frac{1}{\sqrt2}\cosh(y\sqrt2)\cosh y,\
\frac{1}{\sqrt2}\cosh y,\ \qquad \vspace{3mm} \\
& & \qquad \cosh(y\sqrt2)\sinh y-\frac{1}{\sqrt2}\sinh(y\sqrt2)\cosh y\Big)
\end{array}
\end{equation*}
\begin{equation*}
\begin{array}{lll}
B(y)&=&\Big(\sinh(y\sqrt2)\cosh y-\frac{1}{\sqrt2}\cosh(y\sqrt2)\sinh y,\
\frac{1}{\sqrt2}\sinh y,\ \qquad \vspace{3mm} \\
& & \qquad \cosh(y\sqrt2)\cosh y-\frac{1}{\sqrt2}\sinh(y\sqrt2)\sinh y\Big).
\end{array}
\end{equation*}
\end{example}
\begin{example}\rm
\label{dmn:parabola}
Concerning the last case in Theorem~\ref{dmn:thm3},
let us choose for example $c_1=(0,1,0)$, $c_2=(0,0,1)$ and $c_3=(1,0,0)$.
It follows that
the parametrization is given by
$$
F(x,y)=\Big(
A(y)\sinh \Big(\int_0^x\cos\theta(\tau)d\tau\Big)+ B(y)\cosh \Big(\int_0^x\cos\theta(\tau)d\tau\Big),
$$
$$
 \int_0^x\sin\theta(\tau)d\tau\ \Big)
$$
{\rm where} $\displaystyle A(y)=\Big(y,\ 1-\frac{y^2}{2},\ \frac{y^2}{2}\Big)$ {\rm and}
$\displaystyle B(y)=\Big(y,\ -\frac{y^2}{2},\ 1+\frac{y^2}{2}\Big)$.
\end{example}
An interesting situation occurs when we consider $\theta(x)=x^2$.
In this case, the surface is
$$
  F(x,y) = \left(A(y)\sinh \Big(\sqrt{\frac{\pi}{2}}\ C\Big(\sqrt{\frac{2}{\pi}}\ x\Big)\Big)
+ B(y)\cosh \Big(\sqrt{\frac{\pi}{2}}\ C\Big(\sqrt{\frac{2}{\pi}}\ x\Big)\Big),\ \right.
$$
$$
\left. \sqrt{\frac{\pi}{2}}\ S\Big(\sqrt{\frac{2}{\pi}}\ x\Big) \right)
$$
where {\em C} and {\em S} are the traditional notations for the Fresnel integrals
$\displaystyle C(z)=\int_0^z\cos\left( \frac{\pi t^2}{2}\right)dt$ respectively
$\displaystyle S(z)=\int_0^z\sin\left( \frac{\pi t^2}{2}\right)dt$. The curve involved in the classification theorem is,
in this case, given by $(\phi(x), \chi(x))=(C(x),\ S(x))$, known as {\em Cornu spiral}.

Inspired by the way of writing the parametrization of a surface in
$\H^2\times\R$ with a principal direction $T$ everywhere starting with
two curves in $\S^2_1$ respectively in $\H^2$,
we would like to reformulate the classification Theorem~\ref{dmn:thm4}  using
just one curve as follows

\begin{theorem}
\label{dmn:thm5}
Let $F:M\rightarrow\H^2\times\R$ be an isometrically immersed surface $M$ in $\H^2\times\R$,
with $\theta\neq0, \frac\pi 2$. Then $M$ has $T$ as a principal direction if and only if
$F$ is given, up to rigid motions of the ambient space, by
\begin{equation}
\label{dmn:ft5}
F(x,y)=\Big( f(y)\cosh\phi(x)+ N_f(y)\sinh\phi(x), \chi(x) \Big)
\end{equation}
where $f(y)$ is a regular curve in $\H^2$ and
$N_f(y)=\frac{f(y)\boxtimes f'(y)}{\sqrt{\<f'(y),f'(y)\>}}$ represents the normal of $f$.
Moreover, $(\phi,\chi)$ is a regular curve in $\R^2$ and the angle
function $\theta$ of this curve is the same as the angle function of
the surface parametrized by $F$.
\end{theorem}

\begin{proof}
In the classification Theorem~\ref{dmn:thm4}, if $B$ is not a constant (timelike) vector, rename it by $f$
and hence $f\in\H^2$ with $\<f',f'\> >0$. The curve $A$, from the same theorem, lies on $\S_1^2$ and it is
orthogonal to $B$. Moreover $A'$ and $B'$ are parallel. This yields that $A$ can be identified by
$A=\pm \frac{1}{\sqrt{\<f'(y),f'(y)\>}}~f(y)\boxtimes f'(y)$ and hence parametrization \eqref{dmn:ft5} is obtained.

Suppose $B$ is a constant vector. Adding a constant to $\phi$,
we consider the parametrization
$$
F(x,y)=\left(\widetilde{A}(y)\sinh\tilde\phi(x)+\widetilde{B}(y)\cosh\tilde\phi(x),\ \chi(x)\right)
$$
as in Remark~\ref{rmk:changeAB}.
Now $\widetilde{B}'$ is different from $0$ and hence we can proceed as in the previous case.

The converse part follows from direct computations proving that indeed $T$ is a principal direction of $M$.

\end{proof}

\begin{remark}\rm
If in the previous theorem the angle function $\theta$ is constant, we recover Theorem $3.2$
on the classification of constant angle surfaces in $\H^2\times\R$ in \cite{dmn:DM09}. See also \cite{dmn:MT} for an alternative proof.
\end{remark}


Another classical problem is the minimality of surfaces.
In the following we will investigate this property for surfaces having $T$ as a principal direction.

We give first an auxiliary result useful in the proof of the theorem.
\begin{lemma}           
\label{dmn:lemma3}
Let
\begin{equation*}
\f''-2\cot\f\ \f'^2+\cos\f\sin\f=0
\end{equation*}
be a second order ODE where $\f$ is a smooth function of one variable $x$. Then, the
solution is given by
$\f=\arctan\left(\frac{1}{a(x)}\right)$ where
\linebreak
$a(x)=c_1\cosh x + c_2\sinh x$, $(c_1, c_2\in \R)$ never vanishes.
\end{lemma}
We can state now the following classification theorem.
\begin{theorem}
\label{dmn:thm6}
Let $M$ be a surface isometrically immersed in $\H^2\times\R$, with $\theta\neq0, \frac\pi 2$. Then $M$ is minimal with $T$ as principal
direction if and only if the immersion is, up to isometries of the ambient space, locally given by one of the
following cases

$\hspace{3mm}\quad\quad\quad  F:M\longrightarrow\H^2\times\R$
\scriptsize
\begin{subequations}
\renewcommand{\theequation}{\theparentequation .\alph{equation}}
\label{dmn:t6_1}
\begin{eqnarray}
&& \label{dmn:t6_1a}
F(x,\ y)=\left(\frac{b(x)}{\sqrt{1+c_1^2-c_2^2}}~,\ \frac{\sqrt{a^2(x)+1}}{\sqrt{1+c_1^2-c_2^2}}~\sinh y,\
\frac{\sqrt{a^2(x)+1}}{\sqrt{1+c_1^2-c_2^2}}~\cosh y,\ \chi(x)\right)\\[2mm]
&&\label{dmn:t6_1b}
F(x,\ y)=\left(\frac{\sqrt{a^2(x)+1}}{\sqrt{c_2^2-c_1^2-1}}~\cos y,\ \frac{\sqrt{a^2(x)+1}}{\sqrt{c_2^2-c_1^2-1}}~\sin y,\
\frac{b(x)}{\sqrt{c_2^2-c_1^2-1}}~,\ \chi(x)\right)\\[2mm]
&&
\label{dmn:t6_1c}
F(x,\ y)=\left(b(x)~y,\ \frac{b(x)}{2}~(1-y^2)-\frac{1}{2b(x)},\ \frac{b(x)}{2}~(1+y^2)+\frac{1}{2b(x)},\ \chi(x)\right)
\end{eqnarray}
\end{subequations}
\normalsize
where
\begin{equation}
\label{dmn:t6_2}
\chi(x)=\displaystyle\int\limits_0\limits^x\frac{1}{\sqrt{a^2(\tau)+1}}\ d\tau\
\end{equation}
with $a(x)=c_1\cosh x + c_2\sinh x$, $b(x)=a'(x)$  and $c_1,c_2\in \R$ such that all quantities involved in previous
expressions are well defined.
\end{theorem}
\begin{remark}\rm
In all three cases, the curve $(\phi(x),\ \chi(x))$ is determined up to some real constants $c_1$ and $c_2$ by the angle function
$\theta=\arctan\left(\frac{1}{a(x)}\right)$, where $a(x)=c_1\cosh x + c_2\sinh x$ with $c_1^2+c_2^2\neq 0$.
In particular, in each case of the previous theorem we have
\small
\begin{subequations}
\renewcommand{\theequation}{\theparentequation .\alph{equation}}
\label{dmn:t66_1}
\begin{eqnarray}
&& \label{dmn:t66_1a}
\sinh\phi(x)=\frac{b(x)}{\sqrt{1+c_1^2-c_2^2}}~,\ \cosh\phi(x)=\frac{\sqrt{a^2(x)+1}}{\sqrt{1+c_1^2-c_2^2}}~,\ \\
&& A(y)=(1,\ 0,\ 0)\ {\rm and}\ B(y)=(0,\ \sinh y,\ \cosh y)\nonumber\\[2mm]
&&\label{dmn:t66_1b}
\sinh\phi(x)=\frac{\sqrt{a^2(x)+1}}{\sqrt{c_2^2-c_1^2-1}}~,\ \cosh\phi(x)=\frac{b(x)}{\sqrt{c_2^2-c_1^2-1}}~,\ \\
&& A(y)=(\cos y,\ \sin y,\ 0)\ {\rm and}\ B(y)=(0,\ 0,\ 1) \nonumber\\[2mm]
&&
\label{dmn:t66_1c}
\phi(x)=\pm\ln b(x)~,\\
&& A(y)=\left(y,\ 1-\frac{y^2}{2},\ \frac{y^2}{2}\right)\ {\rm and}\ B(y)=\left(y,\ -\frac{y^2}{2},\ 1+\frac{y^2}{2}\right).
\nonumber
\end{eqnarray}
\end{subequations}
\normalsize

\end{remark}

\begin{proof}[Proof of Theorem $6$]
We choose canonical coordinates $x$ and $y$ as in Proposition~$\ref{dmn:prop3}$.
Starting with the classification Theorem~\ref{dmn:thm4}, the isometric immersion is given by
$$
F(x,y)=\left(A(y)\sinh\phi(x)+B(y)\cosh\phi(x), \chi(x)\right)
$$
where $A(y)$ is a curve in $\S_1^2$, $B(y)$ is a curve in $\H^2$, such that
$\langle A,B\rangle=0$ and $A'|| B'$. Recall that $\phi'(x)=\cos\theta(x)$ and $\chi'(x)=\sin\theta(x)$.
The metric of this surface is given by \eqref{dmn:p4_1}, where
\begin{equation}
\label{dmn:tt6_1}
\beta^2(x,y)=
\frac{1}{2}\big(\<A',A'\>+\<B',B'\>\big)\cosh 2\phi(x)+ \qquad
\end{equation}
$$
           \qquad + \<A',B'\>\sinh2\phi(x)+ \frac{1}{2}\big(-\<A',A'\>+\<B',B'\>\big).
$$
The minimality condition yields
$\displaystyle \theta_x+\tan\theta\frac{\b_x}{\b}=0$. Integrating it once we obtain
$\b=\frac{1}{\sin\theta}
$ 
after a change of $y$-coordinate.
Recall that $\beta$ and $\theta$ are also related by the general equation \eqref{dmn:p4_3}.
Substituting here the expression of $\b$, we
get that the angle function $\theta$ satisfies
$$\theta_{xx}-2\cot\theta\theta_x^2+\cos\theta\sin\theta=0
$$
which has the general solution given by Lemma~\ref{dmn:lemma3}, namely
\begin{equation*}
\label{dmn:tt6_3}
\theta=\arctan\left(\frac{1}{a(x)}\right)
\ {\rm{where}}\ a(x)=c_1\cosh x + c_2\sinh x\  (c_1, c_2\in \R).
\end{equation*}

Next, we immediately compute $\sin\theta=\frac{1}{\sqrt{a^2(x)+1}}$. Hence, the fourth component of the parametrization $F$,
i.e. $\chi(x)$, is given by \eqref{dmn:t6_2} and $\b=\sqrt{a^2(x)+1}$.
Moreover, $\cos\theta=\frac{b'(x)}{\sqrt{b^2(x)+c_1^2-c_2^2+1}}$.

In order to determine explicitly $\phi(x)=\int\cos\theta(x)~ dx$ we distinguish the following cases.

{\bf Case 1.} If $c_1^2-c_2^2+1>0$, then $\phi(x)= {\rm arcsinh}\left(\frac{b(x)}{\sqrt{c_1^2-c_2^2+1}}\right)$, yielding that
$\cosh\phi(x)$ and $\sinh\phi(x)$ are obtained as in \eqref{dmn:t66_1a}.
Combining \eqref{dmn:tt6_1} with $\b=\sqrt{c_1^2-c_2^2+1}\cosh\phi$,
one gets that  the curves $A$ and $B$ must satisfy the following
conditions
$$
\<A',A'\>=0,\ \<B',B'\>=c_1^2-c_2^2+1,\ \<A',B'\>=0.
$$
At this point we conclude that $A$ is a constant curve on $\S_1^2$.
Hence, up to some constants and after a change of $y-$coordinate $y\sqrt{1+c_1^2-c_2^2}\equiv y$,
one can  choose them as in \eqref{dmn:t66_1a} obtaining parametrization \eqref{dmn:t6_1a}.\\
{\bf Case 2.} If $c_1^2-c_2^2+1<0$, then $\phi(x)= {\rm arccosh}\left(\frac{b(x)}{\sqrt{c_2^2-c_1^2-1}}\right)$. Following the same idea as
in the previous case we get similar conditions that must be satisfied by the curves $A$ and $B$:
$$
\<A',A'\>=c_2^2-c_1^2-1,\ \<B',B'\>=0,\ \<A',B'\>=0.
$$
Hence, $B$ is a constant curve on $\H^2$. As in Case 1, curves $A$ and $B$ can be taken as in \eqref{dmn:t66_1b}. With these considerations
\eqref{dmn:t6_1b} is obtained.\\
{\bf Case 3.} If $c_1^2-c_2^2+1=0$, $\phi(x)=\pm\ln b(x)$. In this case
the curves $A$ and $B$ must satisfy
$$
\<A',A'\>=1,\ \<B',B'\>=1,\ \<A',B'\>=\pm 1
$$
and up to some constants, can be chosen as in \eqref{dmn:t66_1c} obtaining parametrization \eqref{dmn:t6_1c}.

{\bf Conversely}, it can be proved that parametrizations \eqref{dmn:t6_1} furnish a minimal surface having $T$ as
principal direction.

\end{proof}

\begin{remark}\rm
When one of the two constants vanishes, the $4^{\rm th}$ component
of the parametrization, in the previous theorem, can be rewritten
using elliptic functions as follows:

\qquad $ \chi(x)=\frac{1}{\sqrt{c^2+1}}\
F\left(\arccos\frac{1}{\cosh x }\ \big|\ \frac{1}{1+c^2}\right)
{\rm\ if\ } c_1=c {\rm\ and\ } c_2=0 $

\qquad $ \chi(x)=\ \ F\left(\arccos\frac{1}{\cosh x }\ \big|\
1-c^2\right) {\rm\ if\ } c_1=0 {\rm\ and\ } c_2=c $

where {\em F $(z\ |\ m)$}=$\displaystyle
\int_0^z\frac{dt}{\sqrt{1-m\sin^2t}}$ is the elliptic integral of
the first kind.
\end{remark}

Concerning the flatness property for surfaces having $T$ as a principal direction, we give
the following classification theorem.


\begin{theorem} 
\label{dmn:thm7}
Let $M$ be an isometrically immersed surface in $\H^2\times\R$, with $\theta\neq0, \frac\pi 2$. Then $M$ is flat with $T$ a principal
direction if and only if the immersion $F$ is, up to isometries of the ambient space, given by

$\hspace{12mm}\quad\quad\quad  F:M\longrightarrow\H^2\times\R$
\small
\begin{subequations}
\renewcommand{\theequation}{\theparentequation .\alph{equation}}
\label{dmn:t7_1}
\begin{eqnarray}
&&
\label{dmn:t7_1a}
F(x,\ y)=\left(\frac{x}{\sqrt{c+1}}\cos y,\ \frac{x}{\sqrt{c+1}}\sin y,\ \frac{\sqrt{x^2+c+1}}{\sqrt{c+1}},\ \chi(x)\right)\\[2mm]
&&\label{dmn:t7_1b}
F(x,\ y)=\left(\frac{\sqrt{x^2+c+1}}{\sqrt{-c-1}},\ \frac{x}{\sqrt{-c-1}}\sinh y,\ \frac{x}{\sqrt{-c-1}}\cosh y,\ \chi(x)\right)\ \\
&& \label{dmn:t7_1c}
\displaystyle
F(x,\ y)=\left(x y,\ \frac{x}{2}(1-y^2)-\frac{1}{2x},\ \frac{x}{2}(1+y^2)+\frac{1}{2 x},\ \chi(x)\right)
\end{eqnarray}
\end{subequations}
\normalsize
where
\begin{equation}
\label{dmn:t7_2}
\displaystyle\chi(x)=\int\limits^x\frac{\sqrt{\tau^2+c}}{\sqrt{\tau^2+c+1}}\ d\tau,\ c\in \R.
\end{equation}
\end{theorem}

\begin{proof}
We use canonical coordinates and start with the parametrization given in Theorem~\ref{dmn:thm4}. Under the flatness condition we obtain that
$
\tan\theta\theta_x\b_x-\cos^2\theta\b=0.
$
Combining it with the PDE \eqref{dmn:p4_3} we obtain
$\b={\mathfrak{a}}(y)x+{\mathfrak{b}}(y)$ where ${\mathfrak{a}}$ and ${\mathfrak{b}}$ are smooth
functions on $M$ depending only on $y$ with ${\mathfrak{a}}$ nowhere vanishing, and such that
\begin{equation}
\label{dmn:tt7_1}
\frac{{\mathfrak{b}}(y)}{{\mathfrak{a}}(y)}=\frac{\tan\theta\theta_x-x \cos^2\theta}{\cos^2\theta}\ .
\end{equation}
Since the right hand side of \eqref{dmn:tt7_1} depends only on $x$, it follows that both sides are
equal to the same constant, say
$c_0\in\R$. Thus
\begin{eqnarray}
\label{dmn:tt7_2}
\b={\mathfrak{a}}(y)(x+c_0)\\[2mm]
\label{dmn:tt7_3}
\cos^2\theta(x+c_0)-\tan\theta\theta_x=0.
\end{eqnarray}
One can change $(x,y)-$coordinates in \eqref{dmn:tt7_2} such that $\b=x$,
but with no effect on the other formulas in Proposition~\ref{dmn:prop3}.
Then from \eqref{dmn:tt7_3} one finds
\begin{equation}
\label{dmn:tt7_5}
\theta=\arctan(\sqrt{x^2+c}),\ c\in \R.
\end{equation}

By direct computations, $\sin\theta=\frac{\sqrt{x^2+c}}{\sqrt{x^2+c+1}}$ and $\cos\theta=\frac{1}{\sqrt{x^2+c+1}}$.
Hence, the last component of the parametrization is given by \eqref{dmn:t7_2}. Let us distinguish the following cases for the real constant $c$.

{\bf Case 1)} $c \geq 0 $.
The solution \eqref{dmn:tt7_5} for $\theta$ is well defined for $x\in (0,\ +\infty)$.
As $c\geq 0$, one gets that $c+1>0$ and
so, $\phi(x)={\rm arcsinh}\left(\frac{x}{\sqrt{c+1}}\right)$. This yields
$\cosh\phi(x)=\frac{\sqrt{x^2+c+1}}{\sqrt{c+1}}$. 
Combining both $\b=\sqrt{c+1}~\sinh\phi$ and \eqref{dmn:tt6_1} it follows that
the curves $A$ and $B$ must satisfy
$$
\<A',A'\>=c+1,\ \<B',B'\>=0,\ \<A',B'\>=0.
$$
Therefore $B$ is a constant curve in $\H^2$. Up to some constants and after a change of the $y-$coordinate
$y\sqrt{c+1}\equiv y$, one can choose
\linebreak
$A=(\cos y,\ \sin y,\ 0)$ and $B=(0,\ 0,\ 1)$ obtaining \eqref{dmn:t7_1a}.

{\bf Case 2)} $c<0$.\\
{\bf Case 2.a)} $c\in (-1,\ 0)$. The angle function \eqref{dmn:tt7_5} is well defined for $x\in(\sqrt{-c},\ +\infty)$. We get again
$c+1>0$ and the rest of the computations are the same as in {\bf Case 1)} occurring also in this case the parametrization \eqref{dmn:t7_1a}.\\
{\bf Case 2.b)} $c<-1$. Using similar arguments as in the previous cases the domain of $x$ is $(\sqrt{-c},\ +\infty)$.
Moreover, $\phi(x)={\rm arccosh}\left(\frac{x}{\sqrt{-c-1}}\right)$. The curves $A$ and $B$ fulfill
$$
\<A',A'\>=0,\ \<B',B'\>=-c-1,\ \<A',B'\>=0.
$$
Hence $A$ is a constant curve in $\S_1^2$, and consequently,  one gets the parametrization \eqref{dmn:t7_1b}.\\
{\bf Case 2.c)} $c=-1$.
From \eqref{dmn:tt7_5} one obtains that $x\in (1,\ +\infty)$. Finally, $\phi(x)=\pm\ln(x)$ and $\b=e^{\pm\phi(x)}$.
Hence, the curves $A$ and $B$ verify
$$
\<A',A'\>=1,\ \<B',B'\>=1,\ \<A',B'\>=\pm 1.
$$
Using the same technique, the parametrization \eqref{dmn:t7_1c} is obtained.

{\bf Conversely}, it is not difficult to prove that the parametrizations \eqref{dmn:t7_1} indeed give a flat surface in
$\H^2\times\R$ with $T$ as a principal direction.
\end{proof}

In the end of this paper, we would like to give some results concerning
the constancy of the mean curvature for surfaces with canonical principal directions.
\begin{proposition}\rm
Let $M$ be a surface in $\H^2\times\R$, with $\theta\neq0, \frac\pi 2$ and $T$ as principal direction. If $M$ has constant principal
curvatures, then $\theta$ is constant and $M$ is given by Example \ref{dmn:parabola} with constant $\theta$.
\end{proposition}
\begin{proof}
We will look for such surfaces by using canonical coordinates.
As a consequence of \eqref{dmn:p4_2} we have
$$
\theta_x=\kappa_1 \ {\rm and\ }
\tan\theta~\frac{\beta_x}{\beta}=\kappa_2,\ {\rm where}\
\kappa_1,\kappa_2\in\R. \vspace{-2mm}
$$
If $\kappa_1=0$ then $M$ is a constant angle surface and we retrieve the Case 3 of Theorem 3.
The result follows also from \cite[Remark 3.5]{dmn:DM09}.\\
If $\kappa_1\neq0$ then $\theta=\kappa_1x+c$, hence non constant, where $c\in\R$. Therefore,
$\beta=\psi(y)[\sin\theta]^\frac{\kappa_2}{\kappa_1}$, where
$\psi$ is a smooth function on $M$. Moreover, $\theta$ and $\beta$
verify also \eqref{dmn:p4_3} which can be rewritten, if $\kappa_2\neq 0$ and $\kappa_1\neq\kappa_2$, as
$$
\cos^2\theta[\sin\theta]^{\frac{\kappa_2}{\kappa_1}-2}\Big((\kappa_2-\kappa_1)\kappa_2-\sin^2\theta\Big)=0$$
yielding a contradiction with $\theta$ non constant ($\kappa_1\neq0$).\\
If $\kappa_2=0$ or $\kappa_1=\kappa_2$ then $\b_x=\kappa_2\psi(y)\cos\theta$ which implies $\b\cos^2\theta=0$, i.e.
$\theta=\frac\pi2$.

\end{proof}

Yet, there exist non-minimal CMC surfaces in $\H^2\times \R$ having
$T$ as principal direction.
\begin{example}\rm
We give an example of such a surface.

Looking at the expression \eqref{dmn:p4_2}
of the Weingarten operator $A$, we write $\partial_x\big(\beta\sin\theta\big)=2H\beta\cos\theta$,
where $H\neq0$ is the mean curvature of the surface $M$.
We try to find $\theta$ and $\beta$ such that $\beta\cos\theta$ is constant.
Integrating the previous equation and taking into account that $\theta$ depends only on $x$ and
$\beta\cos\theta$ is a constant it follows, after a translation in parameter $x$, that
$\theta=\arctan(2Hx)$. Moreover, after a homothetic transformation of the $y$-parameter,
$\beta=\sqrt{1+4H^2x^2}$. As the compatibility
equation \eqref{dmn:p4_3} must be identically satisfied by $\theta$ and $\beta$, we get $H^2=\frac14$. Hence,
$\theta=\arctan(x)$ and $\b=\sqrt{x^2+1}$. Using now the classification Theorem~\ref{dmn:thm4},
$$
F(x,y)=\left(A(y)\sinh\phi(x)+B(y)\cosh\phi(x), \chi(x)\right)
$$
which can be explicitly determined. Namely, $\phi(x)={\rm arcsinh}(x)$ and
the last component is $\chi(x)=\sqrt{1+x^2}-1$.
Combining the expression of $\b=\cosh\phi(x)$ with \eqref{dmn:tt6_1} we get that the curves $A$ and $B$ satisfy
$
\<A',A'\>=0,\ \<A',B'\>=0,\ \<B',B'\>=1
$.
Let us choose $A(y)=(1,0,0)$ and $B=(0,\cosh y, \sinh y)$.
We conclude that the embedding equations of $M$ are
$$
F(x,y)=\left(x,\sqrt{1+x^2}~\sinh y,\sqrt{1+x^2}~\cosh y,\sqrt{1+x^2}-1~\right).
$$
The converse can be proved by straightforward computations.
\end{example}

\smallskip
{\small {\bf Acknowledgement.}
{The authors would like to thank Ruy Tojeiro
for valuable discussions leading to a simplification of Theorem~\ref{dmn:thm5}.}


\end{document}